\documentclass[11pt]{article}
\usepackage{tkz-berge}
\usepackage{tikz}
\usepackage{verbatim}
\usepackage[toc,page]{appendix}
\usepackage{hyperref}
\usepackage{amsmath}
\usepackage{subfigure}
\usepackage{color}
\usepackage{enumerate}
\usetikzlibrary{backgrounds,automata,snakes}
\usepackage{fancyhdr,latexsym,amsmath,amsfonts,amssymb,amsbsy,amsthm,url}
\usepackage[hscale=0.7,centering]{geometry}
\usepackage{graphicx,epsfig}
\usepackage{times}
\usepackage{color}
\usepackage{pstricks}
\usepackage{rotating}
\usepackage[english]{babel}
\usepackage{pgfplots}
\usepackage{mathtools}
\DeclarePairedDelimiter{\diagfences}{(}{)}
\newcommand{\diag}{\operatorname{diag}\diagfences}

\newtheorem{theorem}{Theorem}

\numberwithin{equation}{section}

\makeatletter

\renewcommand{\section}{
  \@startsection
  {section}
  {1}
  {0pt}
  {1.1\baselineskip}
  {0.2\baselineskip}
  {\sc \centering}
}

\makeatother

\theoremstyle{definition}
\newtheorem{defn}{Definition}[section]

\DeclareMathOperator{\tr}{tr}

\usetikzlibrary{shapes.geometric}
\usetikzlibrary{calc,arrows,shapes,snakes,automata,backgrounds,petri}
\title{Ecologically Sustainable Partitioning of a Metapopulations Network}
\author{{Dinesh Kumar{~\thanks{Research financially supported by NBHM, Department of Atomic Energy, Govt. of India.}$~~^,$\thanks{Corresponding author. \newline \indent ~~~Emails: kdinesh@iisc.ac.in (Dinesh Kumar), jatingupta@iisc.ac.in (Jatin Gupta), raha@iisc.ac.in (Soumyendu Raha).}}}~, Jatin Gupta, Soumyendu Raha\\
\\ Department of Computational \& Data Sciences,\\
 Indian Institute of Science, Bangalore 560012, India}
\date{}
\begin{document}
\maketitle
\begin{abstract}
A stable population network is hard to interrupt without any ecological consequences. A communication blockage between patches may destabilize the populations in the ecological network. This work deals with the construction of a safe cut passing through metapopulations habitat such that populations remain stable. We combine the dynamical system stability analysis with graph partitioning algorithms in our approach to the problem. It finds such a safe construction, when one exists, provided the algebraic connectivity of the graph components is stronger than all the spatially local instabilities in the respective components. The dynamics of the populations on the spatially discrete patches (graph nodes) and their spatial communication with other patches is modeled as a reaction-diffusion system. By reversing the Turing-instability idea the stability conditions of the partitioned system  are found to depend on local dynamics of the metapopulations and the Fiedler value of the Laplacian matrix of the graph. This leads to the necessary and sufficient conditions for removal of the graph edges subject to the stability of the partitioned  graph networks. An heuristic bisection graph partitioning algorithm has been proposed and examples illustrate the theoretical result.\\
\newline
\textbf{Keywords:} Ecological Networks, Dynamical System Stability, Graph Partitioning.
\end{abstract}

\section{Introduction}

Distribution of the populations over the range of spatially discrete patches is a fundamental and inseparable aspect of their interactions in the ecological domain. Sets  of  spatially isolated populations which are linked by dispersal of multiple potentially interacting species are called metapopulations \cite{Gilpin91, Wilson92}. Human activities, like, construction of roads, railway lines or fencing affects the dispersal of species among their habitat patches. Construction of new infrastructure in between populations habitat reduces both the quantity and quality of wildlife habitat \cite{Reijnen00}. Large continuous habitats become highly fragmented and leaving small habitat patches dispersed throughout the landscape. Populations in such small, isolated habitats have higher risk to become extinct, and simultaneously recolonization chance is reduced \cite{Opdam91,Hanski99}. Such risks merit the study of existence, and if possible, computation of safe spatial cuts, that is, removal of
links between habitat patches.

Theoretical studies have shown that metapopulations persistence depends upon an interaction between local density, dispersal and spatial
heterogeneity \cite{Levin74, Chesson81, Kareiva90}. In \cite{Amarasekare98} authors examined the influences of density dependent and
independent dispersal in local  dynamics considering spatially heterogeneous environment and mortality during dispersal.
The study found that with logistic dynamics, dispersal changes the strength of it within patches, while Allee dynamics creates between-patch effects.
Role of space and diffusion in the dynamics, stability and persistence of populations is studied by \cite{Bascompte94} and it has been shown that
the larger the spatial domain or diffusion, the more unstable the dynamics is. Connectivity or species dispersal movement in long term survival (stability) of the metapopulations remains a research issue. Dispersal of species in metapopulations network plays the role in
both stabilizing and destabilizing metapopulations and depends on dispersal intensity \cite{Briggs04}. The mode of density-dependence dispersal is the key factor for viability of sources and long-term persistence of source–sink systems \cite{Amarasekare04}. In \cite{Bascompte94, Vance84} it is shown that increased dispersal can destabilize the previously stable systems, whereas studies in \cite{Hassell95, Rohani96, Jang00} concluded that density-independent dispersal does not affect the stability.

The unstable butterfly {\it Melitaea Cinixa} populations in every patch can be stabilized by the dispersal movement (diffusion) of populations among habitat patches in the network; their dispersal among the patches affect the dynamics substantially and existence of alternative stable equilibrium points \cite{Hanski94} are possible. Inspired by this real example, we considered the whole spatially structured populations are stable only if there is dispersal movement of individuals among habitat patches. It is akin to the reverse of the Turing-instability condition \cite{Turing52}.

This work provides an approach to find out appropriate partition(s) \emph{i.e.,} human made cut(s) of an ecological metapopulations network such that the populations remain stable after the partitioning. A combination of qualitative theory of differential equations and graph-theory turn out to be a very useful in providing such desired partition(s).

The role of graph theory is not new in ecology; there are many applications of it in conservation biology and landscape ecology.
A theoretical analysis of stability and persistence of ecological metapopulations especially focusing on a marine system has
been done in \cite{Randrup10} where authors determine the conditions of persistence of metapopulations (a age-structured
patch populations where patch sub-populations are connected by larval dispersal) by graph theoretic methods and finds that among various
factors, migration of individuals between patches is very critical factor for the overall stability of metapopulations. Using graph
theory approach in \cite{Urban01} it is shown that the population can persist despite the substantial losses of habitat area,
as long as there exists a spanning tree.

Metapopulations constitute a complex network and the spectra of the complex network (in particular, second smallest eigenvalue of Laplacian matrix, called the Fiedler value) carries a lot of information about them. For example, second smallest eigenvalue of the Laplacian is associated with the connectivity of the graph (network). Master stability function (MSF) technique, to analyze the stability of synchronized state of coupled oscillators depends upon the ratio of the largest to second-smallest eigenvalue of the Laplacian \cite{Pecora98}. Second smallest eigenvalue of Laplacian emerged as a critical parameter in systems and control. Works \cite{Tanner07, Fax04} in networked dynamic system observed that the  Fiedler value is a measure of stability and robustness. In a similar way, the spectrum of the Laplacian of the graph of certain molecules can be used to predict their chemical properties \cite{Merris94, Mohar91}. Our analysis too, shows the significant role played by Fiedler value in the stability of metapopulations.

The problem of partitioning a graph into smaller components, in science and engineering, has many significant practical importance. One practical example is arranging the electronic components on very large scale integrated circuits in such way that number of connections between two partitions is minimum \cite{Kernighan70}. The present work proposes two ways to partition the graph representing the metapopulations network, so that after partition the stability of the populations remains intact in the smaller sub-networks. One method is the exhaustive procedure that gives a set of ``desired`` partitions and analyses all of them for
the ''best`` choice. The other one is a heuristic bisection algorithm with a local convergence subject to the Fiedler value beyond the the threshold level. Unfortunately, graph partitioning itself is a NP-hard (a complexity class of decision problems that are intrinsically harder than those which can be solved in polynomial time) problem \cite{Garey79}.

The outline of the paper is as follows; in the next section metapopulations' dynamics in a network is considered and stability conditions are explored and analyzed. Section 3 provides the exhaustive graph partitioning procedure along with numerical examples. In section 4, a more useful heuristic graph bisection algorithmis  presented and it's implementation is shown on an Erd\H{o}s-R\'{e}nyi graph.

\section{Model \& Linearized Stability Analysis}
A classic approach to metapopulations dynamics are patch occupancy models, where the fraction of occupied patches is considered and
explicit local population dynamics  is ignored \cite{Levins71}. Within-patch dynamics along with the populations' dispersal among patches has been considered in numerous studies such as \cite{Nakao10, Li10, Ide16, Manyombe17}. The regional persistence of predator–prey  interactions has been considered in patch occupancy models in \cite{Caswell78, McCauley93} and also in models with explicit local dynamics \cite{Crowley81}.

We consider an ecological network of prey and predator populations, where both the populations occupy the spatially discrete patches of the network
and diffuse over the corridors that links them. Each patch is represented as a node of the graph $G$. Hence an ecological network of $n$ patches is a graph with $n$ nodes. At each node of the graph $G$, the dynamics of prey and predator populations is governed by a set of two differential equations  (for example, a Lotka-Volterra model equations). Dispersal movement of the species among patches occurs along the links or corridors represented as the edges of the graph. The amount of dispersal between patches is proportional to difference of populations densities \cite{Nakao10}, and the proportionality constant (i.e., the dispersal rate) is given by the weight on the corresponding edge. Thus the dynamics of predator-prey population in the $i$-th habitat patch with passive migration from/to other patches of the populations is given by the following set of ordinary differential equations \cite{Levin74},

\begin{eqnarray}
\label{model}
\begin{aligned}
\dot{x}_i &=& f_{i}(x_{i}, y_{i})+\sum_{{j=1,j\neq i}}^{n}w^{x}_{ij}(x_{j}-x_{i}),\\
\dot{y}_i &=& g_{i}(x_{i}, y_{i})+\sum_{{j=1,j\neq i}}^{n}w^{y}_{ij}(y_{j}-y_{i}),
\end{aligned}
\end{eqnarray}

where $i=1,2,\dots,n$ and $x_i(t)$ and $y_i(t)$ are the prey and predator densities respectively, at time $t$ in patch $i$. Real-valued functions $f_{i} \in C^1([0,\infty)^2)$ and $g_{i} \in C^1([0,\infty)^2)$ represent the dynamics of prey and predator species respectively within the $i^{\text{th}}$ patch, and they are assumed to be arbitrary non-negative continuous differentiable functions over a feasible domain. The net diffusion rates for both the prey and predator populations between patch $i$ and patch $j$ is assumed to be the same in both the directions, i.e., $w^{x}_{ij}=w^{x}_{ji}, ~w^{y}_{ij}=w^{y}_{ji}$ which means the ecological network is an undirected weighted graph.

We assume that the populations are (locally) stable in the connected network, while there is no population stability (including oscillatory behavior) within an isolated habitat patch. For the instability at any individual graph node, any one of the  conditions  must hold true at each point of $(x_i,y_i)$-state space except $x_i=0$ or $y_i=0$.
\begin{enumerate}
\item $\frac{\partial {f_{i}}}{ {\partial x_{i}}}+\frac{\partial {g_{i}}}{ {\partial y_{i}}}>0.$
\item $\frac{\partial {f_{i}}}{ {\partial x_{i}}}+\frac{\partial {g_{i}}}{ {\partial y_{i}}}<0,~~\frac{\partial {f_{i}}}{ {\partial x_{i}}}\frac{\partial {g_{i}}}{ {\partial y_{i}}}-\frac{\partial {f_{i}}}{ {\partial y_{i}}}\frac{\partial {g_{i}}}{ {\partial x_{i}}}<0.$
\end{enumerate}
The above conditions are a direct consequence of the Dulac-Bendixson Criterion \cite{Kot01}. The second condition needs to
be dropped for diffusion induced stability in spatially structured populations.
The system (\ref{model}) may be written in vector form as follows:
\begin{eqnarray}
\label{model_vector}
\begin{aligned}
\dot{X} &=& F - L^{x}X,\\
\dot{Y} &=& G - L^{y}Y,
\end{aligned}
\end{eqnarray}
where $X=(x_{1},x_{2},\dots,x_{n})^{\top}, Y=(y_{1},y_{2},\dots,y_{n})^{\top}, F=(f_{1},f_{2},\dots,f_{n})^{\top}, G=(g_{1},g_{2},\dots,g_{n})^{\top}$, and $L^{x}$ and $L^{y}$ are Laplacians of the graphs corresponding to prey and predator respectively. For brevity, we assume $L^{x}=L^{y}=L$, that is, both prey and predators are having same diffusion rate between patches that are connected. The Laplacian of the graph is defined as follows,
\[L_{ij}=
\begin{cases}
~~~d_i, &\text{if}~~ i=j,\\
-w_{ij}, &\text{if}~~i\sim j,\\
~~~0, &\text{otherwise}.
\end{cases}
\]
The  weight $d_i=\underset{i\sim j}{\sum}w_{ij}$ is  the  sum  of  the  weights  of  edges  incident  on  node $i$.

Linearizing the system (\ref{model_vector}) around it's nontrivial equilibrium solution, we obtain
\begin{eqnarray*}
\left(\begin{array}{c}\dot{X}\\ \dot{Y}\end{array}\right) &=&
\left[
\begin{matrix}
\frac{\partial F}{\partial X}& \frac{\partial F}{\partial Y}\\
\frac{\partial G}{\partial X}& \frac{\partial G}{\partial Y}
\end{matrix}
\right]
\left(\begin{array}{c}X\\Y\end{array}\right)
-
\left[
\begin{matrix}
L& 0\\
0& L
\end{matrix}
\right]
\left(\begin{array}{c}X\\Y\end{array}\right)
\end{eqnarray*}

Let $X= \sum_j \xi_{1}^j (t)\Phi_{j}$ and $Y=\sum_j \xi_{2}^j (t)\Phi_{j}$, where $\Phi_{j} \in \mathbb{R}^{n}$ is an eigenvector of the Laplacian matrix $L$ of the graph $G$ corresponding to the eigenvalue $\lambda_{j}$, {\it i.e}, $L\Phi_{j}=\lambda_{j}\Phi_{j}$ and $\xi_{i} (t), i=1,2$ is the scalar valued function of time. Then, term wise, the above system can be written as
\begin{eqnarray*}
\left(\begin{array}{c}\dot{\xi}_{1}^j\Phi_{j}\\ \dot{\xi}_{2}^j \Phi_{j} \end{array}\right) &=&
\left[\begin{matrix}
\frac{\partial (f_{1},f_{2},\dots,f_{n})}{\partial (x_{1},x_{2},\dots,x_{n})}&\frac{\partial (f_{1},f_{2},\dots,f_{n})}{\partial (y_{1},y_{2},\dots,y_{n})}\\
\frac{\partial (g_{1},g_{2},\dots,g_{n})}{\partial (x_{1},x_{2},\dots,x_{n})}&\frac{\partial (g_{1},g_{2},\dots,g_{n})}{\partial (y_{1},y_{2},\dots,y_{n})}
\end{matrix}\right]
\left(\begin{array}{c}\xi_{1}^j\Phi_{j}\\\xi_{2}^j\Phi_{j}\end{array}\right)-
\lambda_{j} I_{2n}
\left(\begin{array}{c}\xi_{1}^j\Phi_{j}\\\xi_{2}^j\Phi_{j}\end{array}\right)
\end{eqnarray*}

\begin{eqnarray*}
\left(\begin{array}{c}\dot{\xi}_{1}^j\Phi_{j}\\ \dot{\xi}_{2}^j\Phi_{j} \end{array}\right) &=&
\left[\begin{matrix}
\frac{\partial (f_{1},f_{2},\dots,f_{n})}{\partial (x_{1},x_{2},\dots,x_{n})}-\lambda_{j}I_{n}&\frac{\partial (f_{1},f_{2},\dots,f_{n})}{\partial (y_{1},y_{2},\dots,y_{n})}\\
\frac{\partial (g_{1},g_{2},\dots,g_{n})}{\partial (x_{1},x_{2},\dots,x_{n})}&\frac{\partial (g_{1},g_{2},\dots,g_{n})}{\partial (y_{1},y_{2},\dots,y_{n})}-\lambda_{j}I_{n}
\end{matrix}\right]
\left(\begin{array}{c}\xi_{1}^j\Phi_{j}\\\xi_{2}^j\Phi_{j}\end{array}\right)
\end{eqnarray*}
We define
\begin{eqnarray*}
A&:=&
\left[\begin{matrix}
\frac{\partial (f_{1},f_{2},\dots,f_{n})}{\partial (x_{1},x_{2},\dots,x_{n})}&\frac{\partial (f_{1},f_{2},\dots,f_{n})}{\partial (y_{1},y_{2},\dots,y_{n})}\\
\frac{\partial (g_{1},g_{2},\dots,g_{n})}{\partial (x_{1},x_{2},\dots,x_{n})}&\frac{\partial (g_{1},g_{2},\dots,g_{n})}{\partial (y_{1},y_{2},\dots,y_{n})}
\end{matrix}\right]
=
\left[\begin{matrix}
\diag{\frac{\partial f_{1}}{\partial x_{1}},\frac{\partial f_{2}}{\partial x_{2}},\dots, \frac{\partial f_{n}}{\partial x_{n}}}
& \diag{\frac{\partial f_{1}}{\partial y_{1}},\frac{\partial f_{2}}{\partial y_{2}},\dots, \frac{ \partial f_{n}}{\partial y_{n}}}\\
\diag{\frac{\partial g_{1}}{\partial x_{1}},\frac{\partial g_{2}}{\partial x_{2}},\dots, \frac{\partial g_{n}}{\partial x_{n}}}
& \diag{\frac{\partial g_{1}}{\partial y_{1}},\frac{\partial g_{2}}{\partial y_{2}},\dots, \frac{\partial g_{n}}{\partial y_{n}}}
\end{matrix}\right].
\end{eqnarray*}
For each eigenvalue $\lambda_{j}$ of the Laplacian $L$, let the eigenvalue of the coefficient matrix be denoted by $\sigma_{j}$.
These eigenvalues determine the temporal growth and are given by the roots of the characteristic polynomial of the coefficient matrix
$A-\lambda_{j}I$, {\it i.e.,}
 \[\mbox{det} \left( A-\lambda_{j}I -\sigma_{j} I\right) =0.\]
{\it i.e.}, at the $i$th patch, the following holds:
\begin{equation}
\label{stabilityEqn}
\sigma_{j}^{2}-\sigma_{j}\left(\tr J_{i}-2\lambda_{j}\right)+\left(\lambda_{j}^{2}-\lambda_{j} \tr J_{i}+\det{J_{i}}\right)= 0, ~~1\leq i \leq n.
\end{equation}
where $J_{i}$ is the Jacobian of the reaction terms of the prey and predator populations in $i$th patch calculated at the co-existential equilibrium point of the system (\ref{model}), that is, \[J_{i}:=
\left[\begin{matrix}
\frac{\partial f_{i}}{\partial x_{i}} & \frac{\partial f_{i}}{\partial y_{i}}\\
\frac{\partial g_{i}}{\partial x_{i}} & \frac{\partial g_{i}}{\partial y_{i}}
\end{matrix}\right].
\]
From (\ref{stabilityEqn}), it follows that the conditions
\begin{enumerate}
\item $\lambda_{j}^{2}-\lambda_{j} \tr J_{i}+\det{J_{i}}\geq 0 ~\iff~ (\tr J_{i})^{2}-4\det{J_{i}}\leq 0, ~~1\leq i \leq n$.
\item $\lambda_{j}\geq \frac{1}{2}\underset{i}{\max}\left\{\tr J_{i}\right\} ~\forall~ j>1 ~\iff~\lambda_{2}\geq \frac{1}{2}\underset{i}{\max}\left\{\tr J_{i}\right\}, ~i \in \{1,2,\dots,n\}$.
\end{enumerate}
must hold for the stability of the system (\ref{model}), where $0=\lambda_1\leq\lambda_2\leq\dots \leq \lambda_n $ is the spectrum of the Laplacian matrix $L$ of the graph $G$ induced by the patch network. The above conditions show that for the populations to be stable after the partition of the graph of the patch network, the Fiedler value (second smallest eigenvalue) of the Laplacian of each component must be greater than or equal to the threshold value (condition $2$) and both the eigenvalues of the Jacobian matrix of the reaction term must be complex number with positive real part as $\det J_{i}\geq 0,~\forall ~i$ (conditions $1$), implying that $\tr J_{i}\geq 0$ to have instability in the patch. Also from condition $2$, $\frac{1}{2}\tr J_{i}= \frac{1}{2} \times (\text{sum of positive real part eigenvalues of}~J_{i}) \leq \lambda_2$, it can be inferred that the local instability within a patch must not be greater than that of the algebraic connectivity \emph{i.e.,} $\lambda_2 (L)$ of the graph.

In ecological terms, sustainable separations of the patches are possible provided the populations dynamics inside each patch and populations' dispersal movement among patches are constrained by the above conditions.
The stability conditions depend upon dispersal movement of species (connectivity of the network) and the local populations dynamics. By construction, the dispersal movement of the species in the network is an important stability factor. The analysis thus shows that more the dispersal, the better the stability of the populations. A partition with higher Fiedler value in the components ensures persistence of stable metapopulations.

Stability condition infers that we do not have to require the model's parameters such as growth and mortality rate of the populations. The algebraic connectivity of the graph (dispersal parameters) is enough to conclude about the populations stability in metapopulations network or components.

There are many easily computable lower and upper bounds for the $\lambda_2(L(G))$ in the existing literature \cite{Merris94,Anderson85,Hahn97}. Some of these bounds of $\lambda_2(L(G))$ are listed in the Table~\ref{bounds_t}, where $\delta(G)$ is minimum sum of the weights of edges that are incident on a node;  $E(S,\overline{S}), ~ S\subseteq V(G)$, is the weight on edge cut, $D$ is the diameter of the graph, $\overline{d}$ is the mean distance, and $d_{i}$ is sum of weights of edges that are incident on node $i$. If  $\tau$ is compared with those bounds (see Theorem \ref{thm_ub} and Theorem \ref{thm_lb} below), then the bounds provide an easy way to an otherwise hard to check populations stability on the graph of the patch network.
Table ~\ref{SpecialGraph_t} lists the Fiedler values of some graphs with edges having unit weight.

\begin{theorem}
\label{thm_ub}
If $\tau$ is greater than the upper bound of $\lambda_2(L(C))$, then the partitioned component $C$ is unstable.
\end{theorem}

\begin{theorem}
\label{thm_lb}
If $\tau$ is less than the lower bound of $\lambda_2(L(C))$, then the partitioned component $C$  is stable.
\end{theorem}

\begin{table}
\begin{center}
\begin{tabular}{|c|c|c|c|c|}
\hline
\textbf{Lower Bound} &$\frac{4}{nD}$ & $\frac{4}{(2(n-1)\overline{d}-n+2)}$ & $\underset{i\sim j}{\max} (d_i + d_j)-(n-2)$ \\
\hline
\textbf{Upper Bound }& $\frac{1}{2}\left(\underset{k\in V}\sum a_{ik} +\underset{l\in V}{\sum}a_{jl}\right),$ & $\frac{n}{n-1}\delta(G)$ & $n\frac{E(S,\overline{S})}{|S|(n-|S|)}$ \\
&$i~ \text{and}~ j $ are non-adjacent vertices & & \\
\hline
\end{tabular}
\caption{Various bounds for second-smallest eigenvalue of Laplacian matrix of a graph $G$.}
\label{bounds_t}
\end{center}
\end{table}

\begin{table}
\begin{center}
\begin{tabular}{|c|c|c|c|c|c|c|}
\hline
\textbf{Graph (G)} & Path & Cycle & Cube & Complete & Star & Tree\\
\hline
\textbf{$\lambda_2(G)$} &$2(1-\cos(\frac{\pi}{n}))$&$2(1-\cos(\frac{2\pi}{n}))$&$2$& $n$ & $1$& $\leq 1$\\
\hline
\end{tabular}
\caption{Second-smallest eigenvalue of Laplacian matrix of special graphs.}
\label{SpecialGraph_t}
\end{center}
\end{table}

\begin{theorem}
\label{thm_edgeremoval1}
(\textbf{Necessary Condition for Removal of Some Edges})
Let $G(V,E)$ be a connected and undirected graph with $|V|=n (\geq 3)$. Let $G'$ be a stable subgraph of $G$ obtained by removing some edges. Then the average weight of edges incident on a node of $G'$ i.e., $\langle k\rangle \geq \left(\frac{n-1}{n}\right)\tau$.
\end{theorem}
\begin{proof}
We know that $L_{ii}=d_i=\underset{i\sim j}{\sum}w_{ij}$ is  the  sum  of  the  weights  of  edges  incident  on  node $i$. Thus $\sum_{i=1}^{n}L_{ii}=\sum_{i=1}^{n}d_i$, implies that $$\sum_{i=1}^{n}d_i=\tr L =\lambda_1+\dots+\lambda_n \geq (n-1)\tau,$$ from which it follows that the
average weight of the edges incident on any node must be greater than or equal to $\frac{n-1}{n} \tau$.
\end{proof}

From the stability conditions, it is clear that if the difference between the Fiedler values of the graph Laplacian $L(G)$ and edge-cut induced graph component Laplacian $L(C)$ is not more than the difference between the Fiedler values of $L(G)$ and $\tau$, then the populations in the graph component $C$ will be stable. Although the interlacing theorem places the Laplacian eigenvalues after removal of an edge in a graph in between the two consecutive Laplacian eigenvalues of the original graph, it is hard to estimate the exact difference between these eigenvalues. As a special case, we consider the following:

\begin{theorem}
\label{thm_edgeremoval}
(\textbf{Sufficient Condition for Removal of an Edge}) Let $G(V,E)$ be a connected and undirected graph with $|V|=n (\geq3)$, and let $i$ and $j$ be two adjacent nodes. Let $\varepsilon_i$ be a unit vector in $\mathbb{R}^n$, whose $i$-th component is $1$ and $0$ otherwise. The difference between the $i$-th and $j$-th columns of the Laplacian matrix $L(G)$ of $G$ is a non-zero constant multiple of the vector $\mathbf{v}=\varepsilon_i-\varepsilon_j$
which is an eigenvector of $L(G)$ corresponding to an eigenvalue $\lambda \neq 0$. Then the resultant subgraph $G'$ obtained
by deleting the edge of weight $w_{ij}\neq \lambda$ that links $i$-th and $j$-th node of graph $G$ (which is a rank-one perturbation in a Laplacian matrix of graph $G$), is stable provided $w_{ij}\leq \frac{\lambda(L(G))-\tau}{\mathbf{v}^\top \mathbf{v}}=\frac{\lambda(L(G))-\tau}{2}$.\\

\end{theorem}

\begin{proof}
We know that change in labeling of nodes does not affect the properties of $L$, hence we relabel the node $i$ and $j$ as node $1$ and $2$ respectively.
The Laplacian matrix of a subgraph $G'$ that induced from a graph $G$ by deleting an edge is a symmetric rank-one updated Laplacian matrix of graph $G$. That can be written as follows,
\begin{eqnarray}
L(G')&=&L(G)-
\begin{bmatrix}
w_{12} &-w_{12}&0     &\dots &0\\
-w_{12}&w_{12} &0     &\dots &0\\
\vdots &\vdots &\vdots&\dots &\vdots\\
0      &0     &0     &\dots &0
\end{bmatrix}\nonumber \\
&=&L(G)-w_{12}\mathbf{v}\mathbf{v}^T,
\label{eqn}
\end{eqnarray}
where $\mathbf{v}\mathbf{v}^\top =(\varepsilon_1-\varepsilon_2)(\varepsilon_1-\varepsilon_2)^{\top}$ is a symmetric rank-one perturbation matrix with entries $a_{11}=a_{22}=1, a_{12}=a_{21}=-1$ and rest of the entries are zero. $w_{12}$ is the weight of the edge that connects nodes $1$ and $2$. From equation (\ref{eqn}) it is easy to see that $\mathbf{v}$ is an eigenvector of $L(G')$ and the corresponding eigenvalue is $\lambda-w_{12}\mathbf{v}^T\mathbf{v}$. Thus $\lambda$ is reduced by the $w_{12}\mathbf{v}^T\mathbf{v}$ and by Theorem 2.1 of \cite{Ding07} the rest of the eigenvalues do not change. For stability of the subgraph $G'$, we then require $\lambda(L(G)) - w_{12}\mathbf{v}^T\mathbf{v}\geq \tau$.
\end{proof}

For stability, Theorem \ref{thm_edgeremoval} provides an upper bound on the weight of the edge which to be removed from a graph when the graph Laplacian matrix has the special form of eigenvectors. One example of such graph with its Laplacian having eigenvectors of the form $\varepsilon_i-\varepsilon_j, ~1\leq i \leq n,~ i\neq j$, is a complete graph with equal weight on its edges. Figure (\ref{ss}) shows a complete graph $K_4$ with all edges having equal weight $1$ and its only non-zero eigenvalue is $4$ with multiplicity $3$ and the corresponding eigenvectors are in the form of $(0,-1,0,1)=\varepsilon_4-\varepsilon_2,~(0,-1,1,0)=\varepsilon_3-\varepsilon_2,~(1,-1,0,0)=\varepsilon_1-\varepsilon_2$ and their linear combinations. By Theorem \ref{thm_edgeremoval}, we can delete any edge of the graph $K_4$ safely as long as $\tau \leq 2$.

Also note that for $L(G)$ to have an eigenvector in the form of $\varepsilon_i-\varepsilon_j$, the entry $a_{ii}$ and $a_{jj}$ in $L(G)$ must be equal. This information is helpful in deciding whether the matrix has an eigenvector in the desired form or not. If $a_{ii}\neq a_{jj}$ for any pair $(i,j), ~i \neq j$, then the matrix can not have the eigenvectors of the form $\varepsilon_i-\varepsilon_j$, and if $a_{ii}=a_{jj}$, then we check only the difference between the $i$-th and $j$-th columns.
Theorem \ref{thm_edgeremoval} can be generalized as follows by making use of Theorem 3.1 from \cite{Jding07}.

\begin{theorem}
\label{thm_redgeremoval}
(\textbf{Sufficient Condition for Removal of $r$ Edges}) Let $G(V,E)$ be a connected and undirected graph with $|V|=n (\geq3)$,
and let there be $r$ edges $e(i_k,j_k),~1\leq k\leq r$ that link nodes $i_k$ and $j_k$, respectively. Let $\{\mathbf{v}_k=\varepsilon_{i_k}-\varepsilon_{j_k},~ i_k\neq j_k, ~1\leq k\leq r\}$ be a list of eigenvectors of $L(G)$ that correspond to the  eigenvalues $\lambda_k \neq 0, 1\leq k \leq r$. Then the subgraph $G'$ obtained by deleting the $r$ edges $e(i_k,j_k),~1\leq k\leq r$ of weights $w_{i_{k} j_{k}}\neq \lambda_k$ yields the Laplacian $L(G')$ which is a rank-$r$ perturbation to the Laplacian matrix of the graph $G$ \emph{i.e.,} $L(G')=L(G)-\sum_{k=1}^{r}w_{i_{k} j_{k}}\mathbf{v}_{k}\mathbf{v}_{k}^{\top}$, and subgraph $G'$ is stable if
$$\min\{\upsilon_1, \upsilon_2, \dots, \upsilon_r\}\geq \tau,$$
where $\upsilon_k, 1\leq k\leq r$ are eigenvalues of $r\times r$ matrix $\diag{\lambda_1,\dots, \lambda_r}+(w_{i_{1} j_{1}}\mathbf{v}_1^\top,\dots, w_{i_{r} j_{r}}\mathbf{v}_r^\top) (\mathbf{v}_1,\dots, \mathbf{v}_r).$
\end{theorem}

In Theorem \ref{thm_redgeremoval}, if all the $r$ edges are non-adjacent, then all the pairs of the nodes $(i_k,j_k), ~1\leq k \leq r$ are distinct. That ensure the the list $\{\mathbf{v}_k,~1\leq k \leq r\}$ is a set of orthogonal eigenvectors of $L$. The list of eigenvalues of the rank-$r$ updated matrix \emph{i.e.,} $L(G)-\sum_{k=1}^{r}w_{i_{k} j_{k}}\mathbf{v}_{k}\mathbf{v}_{k}^{\top}$ becomes $\{\lambda_1-2w_{i_{1} j_{1}}, \lambda_2-2w_{i_{2} j_{2}} \dots \lambda_r-2w_{i_{r} j_{r}}, \lambda_{r+1}, \dots, \lambda_n\}$. Thus stability of populations in the graph $G'$ then simply requires  $\underset{k}{\min}\{\lambda_k-2w_{i_{k} j_{k}}, ~1\leq k\leq r\} \geq \tau$.

For particular cases, we can apply two principles (see below) by Merris \cite{Merris98} for deciding the candidate edge to be deleted from the graph safely. We consider an eigenvector corresponding to the eigenvalue $\lambda_2$. Figure \ref{sss} is obtained from Figure \ref{ss} by deleting the edge $e(2,3)$ at the cost of reduction in $\lambda_2(L(K_4))$ by $2$, as $x(2)=-x(3)$ in eigenvector $\mathbf{v}_2$. Similarly, Figure \ref{ssss} is obtained from Figure \ref{sss} by deleting the edge $e(1,4)$ without any change in the $\lambda_2(L)$ as $x(1) =x(4)$ in $\mathbf{v}_2$.\\
~\\
\textbf{Edge Principle:} \emph{Let $G$ be a graph, and $\mathbf{x}$ an eigenvector of $L(G)$ corresponding to the eigenvalue $\lambda$ such that $x(u) = x(v)$ for some adjacent nodes $u$ and $v$. Let $G'$ be the graph obtained by removing the edge $e(u,v)$. Then $\mathbf{x}$ is an eigenvector of $G'$ corresponding to the eigenvalue $\lambda$.}\\
~\\
\textbf{Alternating Principle:}\emph{ Let $G$ be a graph, and $\mathbf{x}$ an eigenvector of $L(G)$ corresponding to the eigenvalue $\lambda$. Let the adjacent vertices $i$ and $j$ of $G$ be such that $x(i)=-x(j) (\neq 0)$. Let $G'$ be the graph obtained by deleting the edges between all such paired vertices $i$ and $j$. Then $\mathbf{x}$ is an eigenvector of $G'$ corresponding to the eigenvalue $\lambda -2$.
}

\begin{figure}
\centering
\subfigure[$\lambda_2(L)$~=~4, ~$v_2$~=~(0,-1,1,0).]
{
\begin{pspicture}(0,-1)(4,4)
\psline[showpoints=true](0.2,0.2)(2,3)
\psline[showpoints=true](2,3)(2,1.5)
\psline[showpoints=true](2,1.5)(0.2,0.2)
\psline[showpoints=true](2,1.5)(3.8,0.2)
\psline[showpoints=true](3.8, 0.2)(2,3)
\psline[showpoints=true](3.8,0.2)(0.2,0.2)
\rput(0,0){$1$}
\rput(4,0){$2$}
\rput(2,3.3){$3$}
\rput(2,1.2){$4$}
\end{pspicture}
\label{ss}
}
\hspace{0.9cm}
\subfigure[$\lambda_2(L)$~=~2,~$v_2$~=~(0,-1,1,0).]
{
\begin{pspicture}(0,-1)(4,4)
\psline[showpoints=true](0.2,0.2)(2,3)
\psline[showpoints=true](2,3)(2,1.5)
\psline[showpoints=true](2,1.5)(0.2,0.2)
\psline[showpoints=true](2,1.5)(3.8,0.2)
\psline[showpoints=true](3.8,0.2)(0.2,0.2)
\rput(0,0){$1$}
\rput(4,0){$2$}
\rput(2,3.3){$3$}
\rput(2,1.2){$4$}
\end{pspicture}
\label{sss}
}
\hspace{0.9cm}
\subfigure[$\lambda_2(L)$~=~2,~$v_2$~=~(0,-1,1,0).]
{
\begin{pspicture}(0,-1)(4,4)
\psline[showpoints=true](0.2,0.2)(2,3)
\psline[showpoints=true](2,3)(2,1.5)
\psline[showpoints=true](2,1.5)(3.8,0.2)
\psline[showpoints=true](3.8,0.2)(0.2,0.2)
\rput(0,0){$1$}
\rput(4,0){$2$}
\rput(2,3.3){$3$}
\rput(2,1.2){$4$}
\end{pspicture}
\label{ssss}
}
\caption{Figure shows the edge deletion by using Fiedler vector's component values and then its impact on Fiedler value.}
\end{figure}
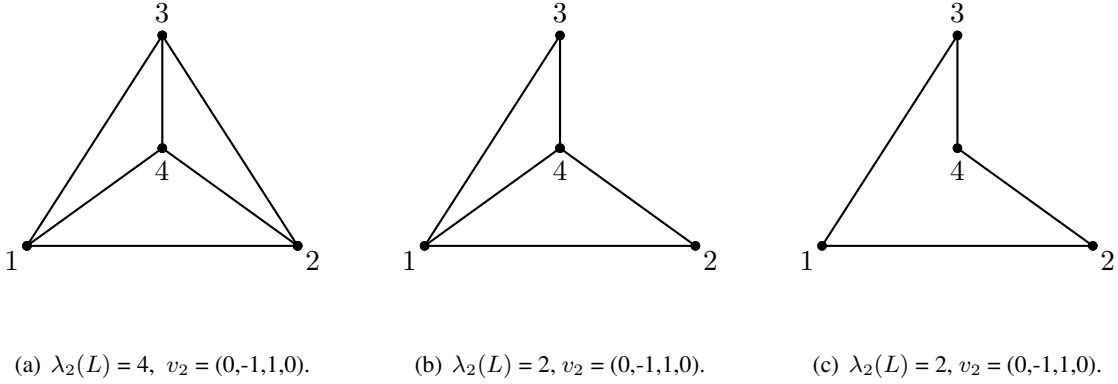

\section{Exhaustive Partitioning}
In this section we present a method exploring all possible partitions such that the stability conditions are satisfied by the graph components after the cut.  Some preliminaries \cite{Deo74} related to this procedure are as follows.
\begin{defn}
\textbf{Cut-Set} of a graph $G$ is a set of edges whose removal from $G$ leaves $G$ disconnected, provided no proper subset of these edges disconnects (in the same way) the graph.
\end{defn}
\begin{defn}
\textbf{Fundamental Cut-Set} with respect to spanning tree $T$, is a cut-set containing exactly one branch of spanning tree $T$.
\end{defn}
\begin{defn}
\textbf{Ring Sum} of two graphs $G_{1}(V_{1},E_{1})$ and $G_{2}(V_{2},E_{2})$ is a graph consisting of the vertex set $V_{1} \cup V_{2}$ and edges that are either in $G_{1}$ or $G_{2}$, but not in both.
\end{defn}
\textbf{Vector space associated with a graph}:
Let us consider the graph $G$ with edges $e_{1},e_{2},\cdots,e_{n}$. Any subset $g$ of these $n$ edges can be represented by a $n-$ tuple $\mathbf{X}=(x_{1},x_{2},\cdots,x_{n})$ such that $x_{i}=1$ if $e_{i}$ is in $g$ and $x_{i}=0$ otherwise.
There exists a vector space $W_{G}$ over Galois field modulo $2$ associated with every graph $G$, where vector addition is taken as ring sum of corresponding graphs, defined as $\mathbf{X} \oplus \mathbf{Y} =(x_{1}+y_{1},x_{2}+y_{2},\dots ,x_{n}+y_{n})$ and scalar multiplication defined as $c\cdot \mathbf{X} =(c\cdot x_{1},c \cdot x_{2},\dots c\cdot x_{n})$.
\begin{defn}
\textbf{Cut-Set Vector} is a vector in $W_{G}$ representing either a cut-set or a union of edge-disjoint cut-sets in $G$.
\end{defn}
\begin{defn}
\textbf{Rank of a graph} is the number of branches in any spanning tree of a connected graph $G$.
\end{defn}
\begin{enumerate}[]
\item Result \textbf{[a]}:  \emph{The set of all cut-set vectors in $W_{G}$ forms a subspace $W_{s}$.}
\item Result \textbf{[b]}:  \emph{The set of cut-set vectors corresponding to the set of fundamental cut-sets with respect to any spanning tree, forms a basis for the cut-subspace $W_{s}$}
\item Result \textbf{[c]}:  \emph{The dimension of the cut set subspace $W_{s}$ is equal to the rank $r$ of the graph, and the number of cut-set vectors (excluding $\mathbf{0}$) in $W_{s}$ ~is~ $2^{r}-1$.}
\item Result \textbf{[d]}:  \emph{The ring sum of any two cut-sets in a graph is a either a third cut-set or an edge-disjoint union of cut-sets.}
\end{enumerate}

With the definitions and results stated as above, we give the steps to obtain the desired partition of the graph such that each resultant component has stable populations:
\begin{itemize}
\item STEP~1: Determine all ~$2^{m-1}-1$~ cut-sets of a graph.
\begin{enumerate}
\item select a spanning tree $T$ of the given connected graph $G$.
\item determine all ~$n-1$~ fundamental cut-sets with respect to spanning tree $T$.
\item generate remaining all cut-sets by applying Result \textbf{[d]}, as Result \textbf{[b]} says that the set of fundamental cut-sets is the basis of cut-subspace $W_{S}$.
\end{enumerate}
\item STEP~2: Reject all the cut-sets which gives the isolated node as a component (as we want stable population in each component).\\
~\\
To decide whether a cut-set gives an isolated patch or not, we consider the vertices that are
connected by an edge corresponding to the first non-zero component in the cut-set vector.
If the first $1$ occurs at the $i$-th place in a cut-set vector, then this corresponds to
the edge $e_{i}$ that connects vertices
$v_{p}$ and $v_{q}$ (say). If all the edges that are incident on either $v_{p}$
or $v_{q}$ fall in the considered cut-set vector, then one or both the vertices are isolated
by the cut-set vector and we drop that cut-set from the list. If not, then we go to the next non-zero component in the
same cut-set vector and repeat the above procedure until all nodes are checked for isolation by this cut-set.
\item STEP~3: For the remaining cut-sets, check whether the second-smallest eigenvalue of the Laplacian of each component satisfies the stability conditions or not.
\item STEP~4: List all the cut-sets that give stable populations components.
\end{itemize}
If we minimize (maximize) the cost of cutting the edge(s), then we choose the cut-set which has the minimum (maximum) weight among all cut-sets obtained in STEP~4. This exhaustive procedure has exponential complexity which increases with the number of nodes in the graph, but does determine all the desired partitions. An
illustrative example follows.

\textbf{Example:} Consider the graph $G(5,6)$ with edge weights as shown in Figure \ref{graphEx}.
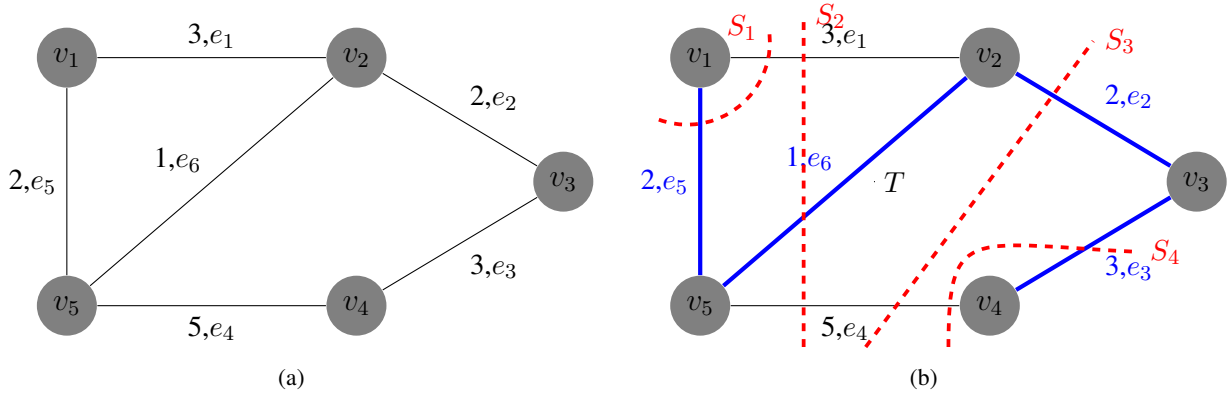
\begin{figure}
\centering
\subfigure[]
{
\begin{tikzpicture}[scale=0.55,node_style/.style={circle,fill=black!50!,font=\sffamily\large\bfseries},
   edge_style/.style={draw=black},auto]
    \node[node_style] (v1) at (-3,2) {$v_1$};
    \node[node_style] (v2) at (4,2) {$v_2$};
    \node[node_style] (v3) at (9,-1) {$v_3$};
    \node[node_style] (v4) at (4,-4) {$v_4$};
    \node[node_style] (v5) at (-3,-4) {$v_5$};
    \draw[edge_style]  (v1) edge node{3,$e_{1}$} (v2);
    \draw[edge_style]  (v2) edge node{2,$e_{2}$} (v3);
    \draw[edge_style]  (v3) edge node{3,$e_{3}$} (v4);
    \draw[edge_style]  (v4) edge node{5,$e_{4}$} (v5);
    \draw[edge_style]  (v5) edge node{2,$e_{5}$} (v1);
    \draw[edge_style]  (v5) edge node{1,$e_{6}$} (v2);
    \end{tikzpicture}
    \label{graphEx}
    }
\hfill
\subfigure[]
{
\begin{tikzpicture} [scale=0.55,node_style/.style={circle,fill=black!50!,font=\sffamily\large\bfseries},
   edge_style/.style={draw=black},auto]
    \node[node_style] (v1) at (-3,2) {$v_1$};
    \node[node_style] (v2) at (4,2) {$v_2$};
    \node[node_style] (v3) at (9,-1) {$v_3$};
    \node[node_style] (v4) at (4,-4) {$v_4$};
    \node[node_style] (v5) at (-3,-4) {$v_5$};
    \draw[edge_style]  (v1) edge node{3,$e_{1}$} (v2);
    \draw[edge_style][ultra thick,blue]  (v2) edge node{2,$e_{2}$} (v3);
    \draw[edge_style][ultra thick,blue]  (v3) edge node{3,$e_{3}$} (v4);
    \draw[edge_style]  (v4) edge node{5,$e_{4}$} (v5);
    \draw[edge_style][ultra thick, blue]  (v5) edge node{2,$e_{5}$} (v1);
    \draw[edge_style][ultra thick, blue]  (v5) edge node{1,$e_{6}$} (v2);
    \draw(1.2,-1)--(1.2,-1)node [right] {$T$};;
    \draw [line width=0.5mm, red, style=dashed ] (-4,0.5) arc (-110:10:2cm)node [left] {$S_{1}$};;
    \draw [line width=0.5mm, red, style=dashed ] (-0.5,-5) -- (-0.5,3) node [right] {$S_{2}$};;
    \draw [line width=0.5mm, red, style=dashed ] (1,-5) -- (6.5,2.4) node [right] {$S_{3}$};;
\draw [line width=0.5mm, red, style=dashed ] (3,-5) .. controls (3,-2) and (3.5,-2.5) .. (7.6,-2.7)node [right] {$S_{4}$};;
\end{tikzpicture}
\label{spanningtree}
}
\caption{Figure (a)~ is a weighted graph $G(5,6)$ and (b) shows a spanning tree $T$ and the fundamental cut-sets w.r.t. it.}
\end{figure}
\begin{itemize}
\item STEP~1: In order to determine all the cut-sets, first we determine the basis elements of the cut-subspace $W_{s}$ of $W_{G}$, which is the set
of fundamental cut-sets with respect to a spanning tree. There are four fundamental cut-sets of the given graph (shown in red dashed lines in Figure \ref{spanningtree}) w.r.t. the spanning tree $T$ (highlighted in blue solid line). Thus a basis of cut-subspace $W_{s}$ is given by $\mathcal{B}=\{S_{1},S_{2},S_{3},S_{4}\}$, where\\
~\\
\begin{tabular}{lllll}
$S_{1}=(1,0,0,0,1,0)^{\top},$ & $S_{2}=(1,0,0,1,0,1)^{\top},$ & $S_{3}=(0,1,0,1,0,0)^{\top},$ & $S_{4}=(0,0,1,1,0,0)^{\top}.$ &~\\
\end{tabular}\\
~\\
Now remaining cut-sets can be generated by taking ring sum (modulo $2$) of these four fundamental cut-sets. The total number of cut-sets generated by fundamental cut-sets is equal to the number of ways ring-sums of fundamental cut-sets can be taken, which is equal to $^4C_2 + ^4C_3 + ^4C_4=11$ and these are given as follows\\
~\\
\begin{tabular}{ll}
1.~~$S_{1} \bigoplus S_{2}=(0,0,0,1,1,1)^{\top},$ &
2.~~$S_{1} \bigoplus S_{3}=(1,1,0,1,1,0)^{\top},$\\
&\\
3.~~$S_{1} \bigoplus S_{4}=(1,0,1,1,1,0)^{\top},$ &
4.~~$S_{2}\bigoplus S_{3}=(1,1,0,0,0,1)^{\top},$\\
&\\
5.~~$S_{2}\bigoplus S_{4}=(1,0,1,0,0,1)^{\top}$, &
6.~~$S_{3}\bigoplus S_{4} =(0,1,1,0,0,0)^{\top},$\\
&\\
7.~~$(S_{1} \bigoplus S_{2})\bigoplus S_{3}=(0,1,0,0,1,1)^{\top}$,&
8.~~$(S_{1} \bigoplus S_{2})\bigoplus S_{4}=(0,0,1,0,1,1)^{\top}$,\\
&\\
9.~~$(S_{1} \bigoplus S_{3})\bigoplus S_{4}=(1,1,1,0,1,0)^{\top},$&
10.~~$(S_{2}\bigoplus S_{3})\bigoplus S_{4}=(1,1,1,1,0,1)^{\top},$\\
&\\
11.~~$(S_{1}\bigoplus S_{2})\bigoplus (S_{3} \bigoplus S_{4})=(0,1,1,1,1,1)^{\top}.$\\
\end{tabular}

~\\
All the cut-sets and components of the graph induced by them are shown in Table 3 (see appendix).\\
\item STEP~2:
Consider the cut set $S_{1}$, first $1$ is at first place, which is associated with the edge
$e_1$ and vertices $v_{1}$ and $v_{5}$ are connected by $e_1$. Note that $e_5$ is the only edge other than $e_1$ that is incident on
vertex $v_{1}$ and it is also in the cut-set $S_1$. Hence cut-set $S_{1}$ separates the vertex $v_{1}$ from the remaining vertices. Hence we drop the cut-set $S_{1}$ from the list of our potential cut-sets. Similarly we reject the cut-sets $S_4,S_{1} \bigoplus S_{2},S_{1} \bigoplus S_{3},$ $S_{1} \bigoplus S_{4},$ $S_{2}\bigoplus S_{3},$ $S_{3}\bigoplus S_{4},$ $(S_{1} \bigoplus S_{3})\bigoplus S_{4},$ $(S_{2}\bigoplus S_{3})\bigoplus S_{4}$ and
$(S_{1}\bigoplus S_{2})\bigoplus (S_{3} \bigoplus S_{4})$.
~\\
\item STEP~3: Below is the list of cut-sets not yielding components that have an isolated patch.\\
~\\
\begin{tabular}{ll}
1. ~$S_{2}=(1,0,0,1,0,1)^{\top},$ & 2.~ $S_{3}=(0,1,0,1,0,0)^{\top},$\\
&\\
3.~ $S_{2}\bigoplus S_{4}=(1,0,1,0,0,1)^{\top},$ & 4.~ $(S_{1} \bigoplus S_{2})\bigoplus S_{3}=(0,1,0,0,1,1)^{\top},$\\
&\\
5.~ $(S_{1} \bigoplus S_{2})\bigoplus S_{4}=(0,0,1,0,1,1)^{\top}$. &\\
\end{tabular}\\
~\\
For each of these cut-sets, the Laplacian matrix pairs corresponding to the components is given as follows,\\
~\\
\begin{tabular}{ll}
1.~~ $\left(
\left[\begin{matrix}2&-2\\-2&2\end{matrix}\right],
\left[\begin{matrix}2&-2&0\\-2&5&-3\\0&-3&3\end{matrix}\right]
\right)$,& 2.~~
$\left(
\left[\begin{matrix}3&-3\\-3&3\end{matrix}\right],
\left[\begin{matrix}4&-3&-1\\-3&5&-2\\-1&-2&3\end{matrix}\right]
\right)$,\\
&\\
3.~~ $\left(
\left[\begin{matrix}2&-2\\-2&2\end{matrix}\right],
\left[\begin{matrix}2&-2&0\\0&5&-5\\-2&-5&7\end{matrix}\right]
\right)$, & 4.~~
$\left(
\left[\begin{matrix}3&-3\\-3&3\end{matrix}\right],
\left[\begin{matrix}3&-3&0\\-3&8&-5\\0&-5&5\end{matrix}\right]
\right)$,\\
&\\
5.~~ $\left(
\left[\begin{matrix}5&-5\\-5&5\end{matrix}\right],
\left[\begin{matrix}3&-3&0\\-3&5&-2\\0&-2&2\end{matrix}\right]
\right)$,
\end{tabular}\\
~\\
and the second-smallest eigenvalues of these Laplacian matrix pairs are given by\\
\begin{tabular}{lllll}
1.~(4,2.35), & 2.~(6,4.26), &3.~(4,2.64), &4.~(6,3.64), &5.~(10,2.35)
\end{tabular}\\
respectively.\\
\item STEP~4:
We need $\frac{1}{2}\underset{i}{\max}\left\{\tr J_{i}\right\}$ to be less than
$\lambda_2(L(G))=3.625$ to ensure the stability of the given network $G(5,6)$.
In particular, if $\frac{1}{2}\underset{i}{\max}\left\{\tr J_{i}\right\}=3$, then the revised list of potential cut-sets that
satisfy both the stability conditions are\\
~\\
\begin{tabular}{ll}
$S_{3}=(0,1,0,1,0,0)^{\top},$ & $(S_{1} \bigoplus S_{2})\bigoplus S_{3}=(0,1,0,0,1,1)^{\top}$.
\end{tabular}
\end{itemize}
~\\
The weight associated with these each cut-sets is $7$ and $5$ respectively.
Thus the maximum weighted cut-set that the system can tolerate is one and only $S_{3}$ qualifies.
If the objective is to minimize the cost of cutting the edges along the sustainable cut-set then
the lowest weighted cut-set is $(S_{1} \bigoplus S_{2})\bigoplus S_{3}$.

Another example we have illustrated is an Erd\H{o}s–R\'{e}nyi graph of 25 vertices, shown in Figure \ref{ErdosRenyi_graph}.
The graph is generated by starting with a set of distinct vertices and adding successive edges between them at random with probability $p=0.5$ \cite{Ide16}. Using exhaustive procedure, the most stable components (corresponding to the maximum possible Fiedler value) of the Erd\H{o}s–R\'{e}nyi graph, are shown in Figure \ref{ErdosRenyi_graphcut}.

\begin{figure}
\includegraphics[scale=0.45]{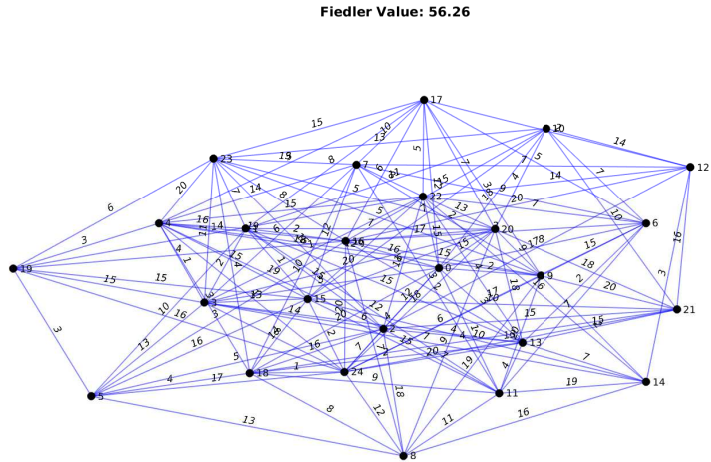}
\caption{A random graph generated by Erd\H{o}s–R\'{e}nyi model with $p=0.5$ and $n=25$, while edge weights distributed uniformly between 1 to 20.}
\label{ErdosRenyi_graph}
\end{figure}

\begin{figure}[h!]
\includegraphics[scale=0.45]{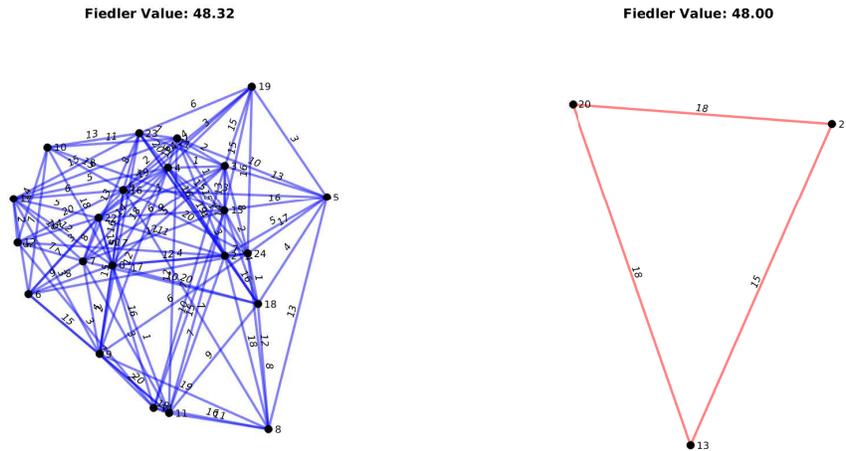}
\caption{Figure shows a graph partitioning of Erd\H{o}s–R\'{e}nyi graph \emph{i.e.,} Fig. \ref{ErdosRenyi_graph} by using exhaustive procedure with maximum possible Fiedler values of both components.}
\label{ErdosRenyi_graphcut}
\end{figure}

\section{Partitioning by Heuristic Bisection Method}
For large networks, the exhaustive partitioning is computationally too expensive. Existing min-cut
algorithms such as spectral bisection partitioning method \cite{Schulz13} and Kernighan-Lin algorithm \cite{Kernighan70}
tend not to cut the edges with more weights and include the edges with minimum weight in the cut-set. It gives the partitioned components higher Fiedler value because of the following result.

\textbf{Weyl's Monotonicity Theorem}\cite{Bhatia97}: \emph{Let $\lambda_j$ be the $j$-th eigenvalue of an $n\times n$ symmetric matrix $A$. If $P$ is a positive semidefinite symmetric matrix, then $\lambda_j(A+P)\geq\lambda_j(A)$,~ $1\leq j\leq n$.}

The above theorem implies that if we increase the weight on an edge of a graph (i.e., perturb the Laplacian with a positive semidefinite symmetric matrix)
then the eigenvalues of the graph's Laplacian will either increase or remain the
same.

Unfortunately the min-cut algorithms do not guarantee well connectedness (large Fiedler value) of components. For example, Figure \ref{spectralcut} shows the spectral bisection of the Erd\H{o}s–R\'{e}nyi graph in Figure \ref{ErdosRenyi_graph} and can be compared with Figure 6 (to be discussed later). In Figure 6, the components have greater Fiedler value than in those in Figure \ref{spectralcut}. The spectral min-cut algorithm cuts the edges with minimum weights but the resultant partitioned components have Fielder values lower than that can be attained.

In this section we provide a graph bi-partitioning algorithm, which starts with a random partition and searches locally around it for an appropriate partitioning that provides stable components. This algorithm also sets up the basis for general partitioning problems such as $k$-way partitioning and partitioning into unequal size components. Motivation for this algorithm is Weyl's Monotonicity Theorem (mentioned above).

Let $G$ be a graph with $n$ nodes, which is to be partitioned into two components $C_{1}$ and $C_2$ with at most one node difference in their sizes (i.e., $||C_1|-|C_2||\leq1$), such that $\lambda_2(L(C_i)), ~i=1,2$ is maximum. Starting with an arbitrary partition of $G$, the appropriate partition can be achieved by pushing the large weighted edges which lie on the graph cut into graph components. This is done by swapping the set of pairs of nodes associated with large weighted edges on the cut, while making sure these nodes take minimal weights away from their respective origin components to the graph cut.

Let $A$ be the set of nodes from component $C_1$ and $B$ be the set of nodes from component $C_2$ by swapping which we are obtain the desired partition.
The following steps identify $A$ and $B$.

\begin{itemize}
\item STEP~1: Determine the difference between the external cost and the internal cost for each node in both the components, defined as the following way:
\[\text{External cost of node}~x: ~~E_x=\underset{y}{\sum}w_{xy},~ ~\text{where}~ x ~\text{and}~ y ~\text{belong to different components.}\]
\[\text{Internal cost of node}~x:~~I_x=\underset{y}{\sum}w_{xy}, ~~\text{where}~ x ~\text{and}~ y ~\text{ both belong to the same component.}\]
\item STEP~2: Select $a_1\in C_1$ and  $b_1\in C_2$ such that $a_1=\underset{s\in C_1}{\max}\{E_s-I_s\}$ and $b_1=\underset{t\in C_2}{\max}\{E_t-I_t\}$.
\item STEP~3: Repeat the STEP~1 and STEP~2 for the components $C_1-\{a_1\}$ and $C_2-\{b_1\}$.
\item STEP~4: Repeat the STEP~3 until  all nodes are exhausted in any of the components.
\item STEP~5: List all  $a_i \in C_1$ and $b_i \in C_2$, that are obtained in the STEP~2.
\item STEP~6: Determine ${\lambda }^{C_1}_{2}(L)\{k\}$ and ${\lambda}^{C_2}_{2}(L)\{k\}$, that is, the second-smallest Laplacian eigenvalue after swapping $k$-pairs of nodes $\{a_1,\dots,a_k \} \in C_1, \{b_1,\dots,b_k\} \in C_2,~ 1\leq k < n/2$.
\item STEP~7: Choose the number $k$ for which ${\lambda }^{C_i}_{2}(L)\{k\}\geq \theta, ~i=1,2$, and swap these $k$-pairs of nodes for the desired partition, where $\theta= \underset{r}{\max}\{\tau +r ~ |~{\lambda }^{C_i}_{2}(L)\{k\}\geq \tau +~r, ~r\in \mathbb{R}\} $.
\end{itemize}

\begin{figure}[h!]
\centering
\includegraphics[scale=0.45]{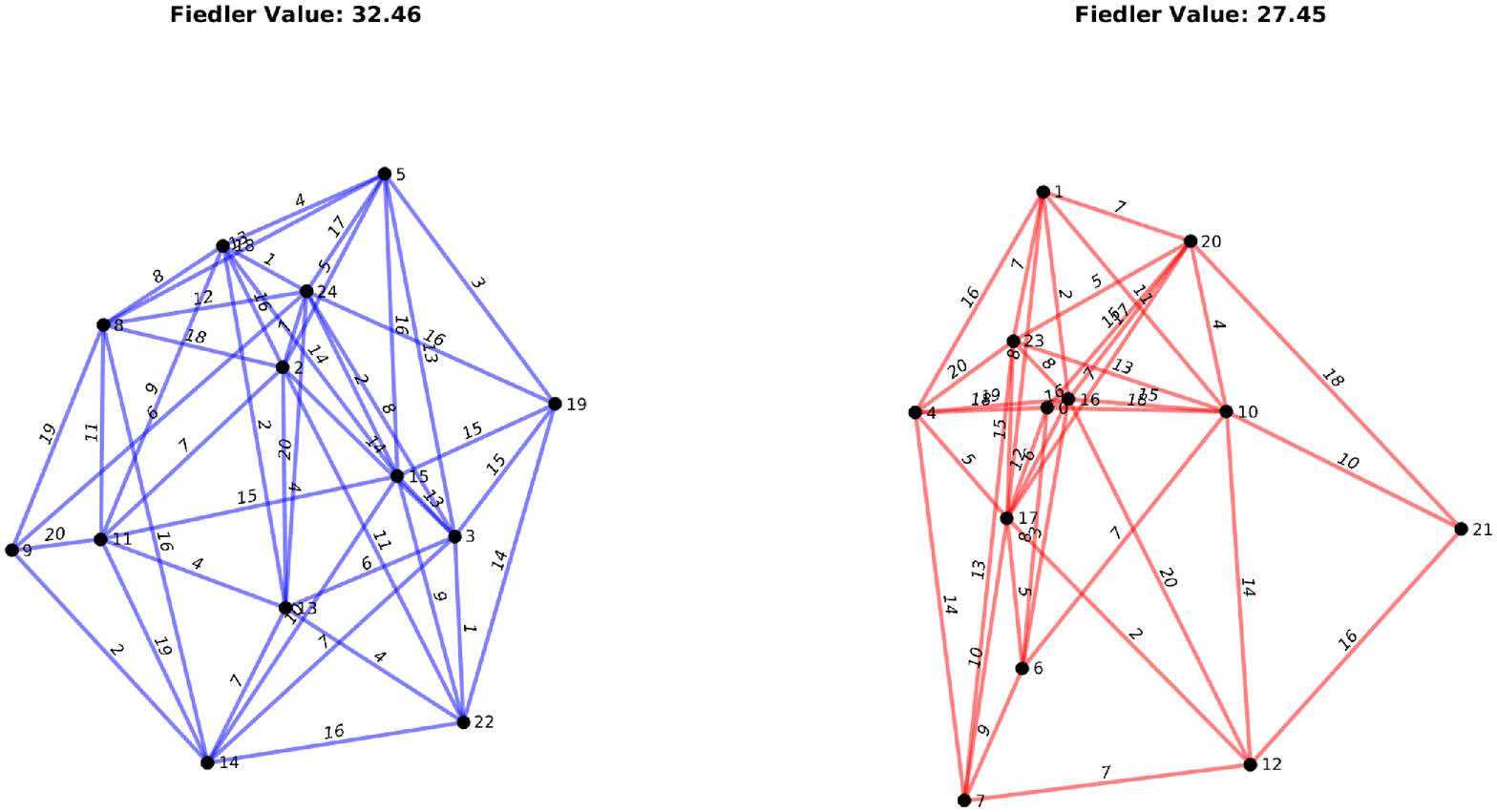}
\caption{Figure shows a graph partition of Erd\H{o}s–R\'{e}nyi graph \emph{i.e.,} Fig. \ref{ErdosRenyi_graph} with spectral bisection algorithm.}
\label{spectralcut}
\end{figure}

This way we can generate ecologically sustainable partitions with this method starting different and randomly chosen initial partitions and can choose
the partition that produces the most stable components. From partitions generated by the heuristic algorithm starting with $100$ initial random partitions of the graph in Figure \ref{ErdosRenyi_graph}, Figure \ref{proposedcut} shows the most stable partition (corresponds to initial partition Figure \ref{proposedcut2}) obtained with $\theta=\tau + r=37.03$, where the threshold value $\tau =28$. Partitioning by the heuristic bisection method produces higher Fiedler valued components than the plain spectral bisection partitioning (Figure \ref{spectralcut}).

\begin{figure}[h!]
\label{Cut_proposed}
\includegraphics[scale=0.45]{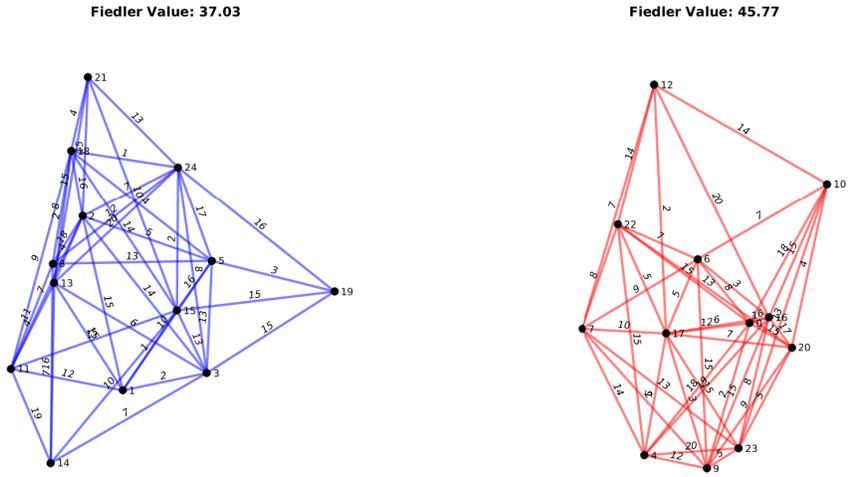}
\caption{Figure illustrates a graph partition of Erd\H{o}s–R\'{e}nyi graph (Fig. \ref{ErdosRenyi_graph}) by using the heuristic bisection algorithm. This partition is the most stable among $100$ trials with different starting partitions.}
\label{proposedcut}
\end{figure}

\begin{figure}[h!]
\label{IniCut_proposed}
\includegraphics[scale=0.45]{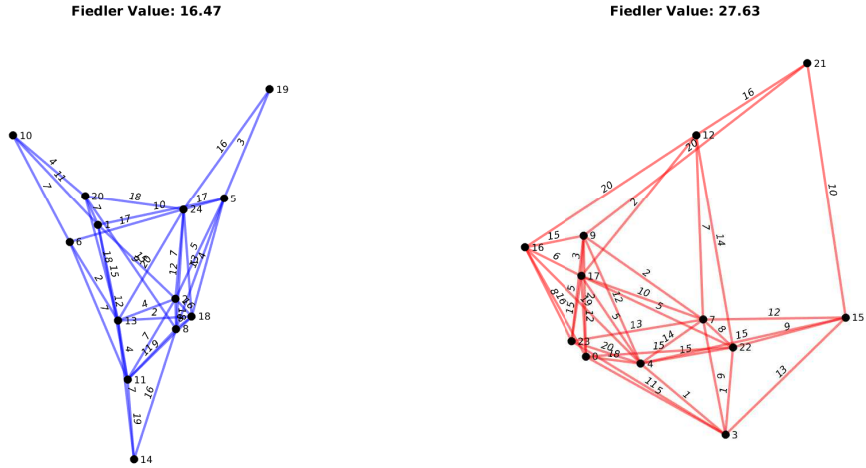}
\caption{Figure shows a random graph partition of Erd\H{o}s–R\'{e}nyi model graph (Fig. \ref{ErdosRenyi_graph}) from which the graph partition Figure 6 is obtained iteratively by swapping pairs of nodes.}
\label{proposedcut2}
\end{figure}

\section{Conclusion \& Discussion}

In this paper we analyzed the stability of ecological metapopulations networks and examined the threshold conditions for the existence of a cut that crosses the metapopulations network without distabilizing the populations. The ecological metapopulations network is considered as a weighted graph $G(V,E)$, the weights on each edge being diectly proportional to the diffusion rate between the nodes (i.e., the habitat patches) connected by the edge. Interaction of two species (predator and prey) on each habitat patch of the network and their diffusion across the habitat patches constitutes the reaction-diffusion
ordinary differential equation model. In a sort of reverse of Turing-instability concept, it has been assumed that the populations are unstable in each of the nodes and the populations are stabilized by the diffusion among the nodes (population patch habitats); else the partitioning problem is meaningless.

Local stability of the metapopulations model governed by the system of reaction-diffusion ordinary differential equations is determined by linearizing the system. There is a significant role of second smallest eigenvalue (Fiedler value) and its importance is observed in deciding the metapopulations stability. The present study finds that the Laplacian Fiedler value must be greater than the threshold value $\frac{1}{2} \underset{i}{\max}\left\{\tr J_{i}\right\}, i=1,2,\dots,n$ for achieving the stability in the spatially distributed populations. In ecological terms, the sustainable separation of the patches into two non-communicative groups (components) is possible provided the populations dynamics inside each habitat patch and populations' dispersal movement among the patches have some appropriate relation. Dispersal movement of species in the network is found to be a crucial factor for stabilizing the populations and thus a high Fiedler value in the partitioned components of the network is important for persistence of the metapopulations.

We have examined an exhaustive procedure and an efficient heuristic bisection algorithm that provides a graph cut (network partition), provided one
exists, such that the metapopulations remain stable even after it is divided into two disconnected and disjoint subnetworks. The exhaustive procedure determines all possible cuts of the graph before deciding an appropriate cut, and the complexity of the procedure increases exponentially with the size of the network. The heuristic bisection algorithm gives the best possible graph cut in the neighborhood of an arbitrary cut (usually obtained by the a random process) from where the algorithm starts and this locally best possible graph cut is achieved by swapping pairs of nodes (loosely, at most $n/2$ nodes) to maximize intra-component Fiedler values. Using different starting cuts for the heuristic bisection algorithm, many ecologically sustainable partitions can be obtained. Having many sustainable partitions we can converge to the most stable partition.

Further work needs to consider directed graphs for inclusion of both sided diffusion and/or taking the predator's diffusion rate dependent on prey gradient. Also transient stability of the partitioning process and time needed for stabilization after the partitioning are other research issues.

\centering
\begin{table}
\begin{tabular}{ |l|l l|c| }
\hline
\textbf{S.No.}& \centering \textbf{Cut-sets} && \textbf{Graph Components}\\
\hline
1.&$S_{1}=(1,0,0,0,1,0)^{\top}$&$=\{e_1,e_5\}$&
\begin{pspicture}(-0.05,-0.05)(1.2,1.2)
\psline[showpoints=true](0,1)(0,1)
\psline[showpoints=true](0.5,1)(1,0.5)
\psline[showpoints=true](1,0.5)(0.5,0)
\psline[showpoints=true](0.5,0)(0,0)
\psline[showpoints=true](0,0)(0.5,1)
\end{pspicture} \\
\hline

2.&{\color{red}$S_{2}=(1,0,0,1,0,1)^{\top}$}&$=\{e_1,e_4,e_6\}$& \begin{pspicture}(-0.05,-0.05)(1.2,1.2)
\psline[showpoints=true](0.5,1)(1,0.5)
\psline[showpoints=true](1,0.5)(0.5,0)
\psline[showpoints=true](0,0)(0,1)
\end{pspicture} \\
\hline
3.&{\color{red}$S_{3}=(0,1,0,1,0,0)^{\top}$}&$=\{e_2,e_4\}$& \begin{pspicture}(-0.05,-0.05)(1.2,1.2)
\psline[showpoints=true](0,1)(0.5,1)
\psline[showpoints=true](1,0.5)(0.5,0)
\psline[showpoints=true](0,0)(0,1)
\psline[showpoints=true](0,0)(0.5,1)
\end{pspicture} \\
\hline
4.&$S_{4}=(0,0,1,1,0,0)^{\top}$&$=\{e_3,e_4\}$&\begin{pspicture}(-0.05,-0.05)(1.2,1.2)
\psline[showpoints=true](0.5,0)(0.5,0)
\psline[showpoints=true](0,1)(0.5,1)
\psline[showpoints=true](0.5,1)(1,0.5)
\psline[showpoints=true](0,0)(0,1)
\psline[showpoints=true](0,0)(0.5,1)
\end{pspicture} \\
\hline
5.&$S_{1} \bigoplus S_{2}=(0,0,0,1,1,1)^{\top}$&$=\{e_4,e_5,e_6\}$& \begin{pspicture}(-0.05,-0.05)(1.2,1.2)
\psline[showpoints=true](0,0)(0,0)
\psline[showpoints=true](0,1)(0.5,1)
\psline[showpoints=true](0.5,1)(1,0.5)
\psline[showpoints=true](1,0.5)(0.5,0)
\end{pspicture} \\
\hline
6.&$S_{1} \bigoplus S_{3}=(1,1,0,1,1,0)^{\top}$&$=\{e_1,e_5\}\cup\{e_2,e_4\}$& \begin{pspicture}(-0.05,-0.05)(1.2,1.2)
\psline[showpoints=true](0,1)(0,1)
\psline[showpoints=true](1,0.5)(0.5,0)
\psline[showpoints=true](0,0)(0.5,1)
\end{pspicture} \\
\hline
7.&$S_{1} \bigoplus S_{4}=(1,0,1,1,1,0)^{\top}$&$=\{e_1,e_5\}\cup\{e_3,e_4\}$& \begin{pspicture}(-0.05,-0.05)(1.2,1.2)
\psline[showpoints=true](0,1)(0,1)
\psline[showpoints=true](0.5,0)(0.5,0)
\psline[showpoints=true](0.5,1)(1,0.5)
\psline[showpoints=true](0,0)(0.5,1)
\end{pspicture} \\
\hline
8.&$S_{2}\bigoplus S_{3}=(1,1,0,0,0,1)^{\top}$&$=\{e_1,e_2,e_6\}$& \begin{pspicture}(-0.05,-0.05)(1.2,1.2)
\psline[showpoints=true](0.5,1)(0.5,1)
\psline[showpoints=true](1,0.5)(0.5,0)
\psline[showpoints=true](0.5,0)(0,0)
\psline[showpoints=true](0,0)(0,1)
\end{pspicture} \\
\hline
9.&{\color{red}$S_{2}\bigoplus S_{4}=(1,0,1,0,0,1)^{\top}$}&$=\{e_1,e_3,e_6\}$& \begin{pspicture}(-0.05,-0.05)(1.2,1.2)
\psline[showpoints=true](0.5,1)(1,0.5)
\psline[showpoints=true](0.5,0)(0,0)
\psline[showpoints=true](0,0)(0,1)
\end{pspicture} \\
\hline
10.&$S_{3}\bigoplus S_{4} =(0,1,1,0,0,0)^{\top}$&$=\{e_2,e_3\}$& \begin{pspicture}(-0.05,-0.05)(1.2,1.2)
\psline[showpoints=true](1,0.5)(1,0.5)
\psline[showpoints=true](0,1)(0.5,1)
\psline[showpoints=true](0.5,0)(0,0)
\psline[showpoints=true](0,0)(0,1)
\psline[showpoints=true](0,0)(0.5,1)
\end{pspicture} \\
\hline
11.&{\color{red}$(S_{1} \bigoplus S_{2})\bigoplus S_{3}=(0,1,0,0,1,1)^{\top}$}&$=\{e_2,e_5,e_6\}$& \begin{pspicture}(-0.05,-0.05)(1.2,1.2)
\psline[showpoints=true](0,1)(0.5,1)
\psline[showpoints=true](1,0.5)(0.5,0)
\psline[showpoints=true](0.5,0)(0,0)
\end{pspicture} \\
\hline
12.&{\color{red}$(S_{1} \bigoplus S_{2})\bigoplus S_{4}=(0,0,1,0,1,1)^{\top}$}&$=\{e_3,e_5,e_6\}$& \begin{pspicture}(-0.05,-0.05)(1.2,1.2)
\psline[showpoints=true](0,1)(0.5,1)
\psline[showpoints=true](0.5,1)(1,0.5)
\psline[showpoints=true](0.5,0)(0,0)
\end{pspicture} \\
\hline
13.&$(S_{1} \bigoplus S_{3})\bigoplus S_{4}=(1,1,1,0,1,0)^{\top}$&$=\{e_1,e_5\}\cup\{e_2,e_3\}$&\begin{pspicture}(-0.05,-0.05)(1.2,1.2)
\psline[showpoints=true](0,1)(0,1)
\psline[showpoints=true](1,0.5)(1,0.5)
\psline[showpoints=true](0.5,0)(0,0)
\psline[showpoints=true](0,0)(0.5,1)
\end{pspicture} \\
\hline
14.&$(S_{2}\bigoplus S_{3})\bigoplus S_{4}=(1,1,1,1,0,1)^{\top}$&$=\{e_1,e_4,e_6\}\cup\{e_2,e_3\}$&\begin{pspicture}(-0.05,-0.05)(1.2,1.2)
\psline[showpoints=true](0.5,1)(0.5,1)
\psline[showpoints=true](1,0.5)(1,0.5)
\psline[showpoints=true](0.5,0)(0.5,0)
\psline[showpoints=true](0,0)(0,1)
\end{pspicture} \\
\hline
15.&$(S_{1}\bigoplus S_{2})\bigoplus (S_{3} \bigoplus S_{4})=(0,1,1,1,1,1)^{\top}$&$=\{e_2,e_5,e_{6}\}\cup\{e_3,e_4\}$& \begin{pspicture}(-0.05,-0.05)(1.2,1.2)
\psline[showpoints=true](0,0)(0,0)
\psline[showpoints=true](1,0.5)(1,0.5)
\psline[showpoints=true](0.5,0)(0.5,0)
\psline[showpoints=true](0,1)(0.5,1)
\end{pspicture} \\
\hline
\end{tabular}
\caption{Table shows all the cut-sets of the graph $G(5,6)$ and their respective components. Cut-sets highlighted in red color shows no isolation of a patch. }
\end{table}


\begin{thebibliography}{99}
\bibitem{Gilpin91} M.E. Gilpin, and I.A. Hanski, Metapopulation Dynamics: Empirical and Theoretical Investigations, \emph{Academic Press, London}, 1991.
\bibitem{Wilson92} D.S. Wilson, Complex interactions in metacommunities, with implications for biodiversity and higher levels of selection, \emph{Ecology}, vol. 73 (1992) 1984–2000.
\bibitem{Reijnen00} R. Reijnen, E. van der Grift, M. van der Veen, M. Pelk, A. Lüchtenborg and D. Bal,  De weg mét de minste weerstand. Opgave Ontsnippering. Expertisecentrum LNV / Alterra, \emph{Green World Research, Wageningen, The Netherlands}, 2000.
\bibitem{Opdam91} P.F.M. Opdam, Metapopulation theory and habitat fragmentation: a review of holarctic breeding bird studies. \emph{Landscape Ecology}, vol. 5 (1991) 93-106.
\bibitem{Hanski99} I. Hanski, Metapopulation Ecology, \emph{Oxford University Press, Oxford, UK}, 1999.
\bibitem{Levin74} S.A. Levin, Dispersion and population interactions, \emph{The American Naturalist}, vol. 108 (1974) 207-228.
\bibitem{Chesson81} P.L. Chesson, Models for spatially distributed populations: the effect of within-patch variability, \emph{Theortical Populations Biology}, vol. 19 (1981) 288-325.
\bibitem{Kareiva90} P.M. Kareiva, Population dynamics in spatially complex environments: theory and data, \emph{Philosophical Transactions of the Royal Society London B}, vol. 330 (1990) 175-190.
\bibitem{Amarasekare98} P. Amarasekare, Interactions between local dynamics and dispersal: insights from single species models, \emph{Theortical Population Biology}, vol. 53 (1998) 44–59.
\bibitem{Bascompte94} J. Bascompte and R.Solé, Spatially induced bifurcations in single-species population dynamics, \emph{The Journal of Animal Ecology}, vol. 63 (1994) 256–264.
\bibitem{Briggs04} C. Briggs and M. Hoopes, Stabilizing effects in spatial parasitoid-host and predator-prey models: a review, \emph{Theortical Population Biology}, vol. 65 (2004) 299–315.
\bibitem{Amarasekare04} P. Amarasekare, The role of density-dependent dispersal in source-sink dynamics, \emph{Journal of Theortical Biology}, vol. 226 (2004) 159–168.
\bibitem{Vance84} R.V. Vance, The effect of dispersal on population stability in one-species, discrete space population growth models, \emph{American Naturalist}, vol. 123 (1984) 23-254.
\bibitem{Hassell95} M. Hassell, O. Miramontes, P. Rohani and R. May, Appropriate formulations for dispersal in spatially structured models: comments on Bascompte \& Solé, \emph{Journal of Animal Ecology}, vol. 64 (1995) 662–664.
\bibitem{Rohani96} P. Rohani, R. May and M. Hassel, Metapopulations and equilibrium stability: the effects of spatial structure, \emph{ Journal of Theoretical Biology}, vol. 181 (1996) 97–109.
\bibitem{Jang00} S. Jang and A. Mitra, Equilibrium stability of single-species metapopulations, \emph{Bullitin of Mathematical Biology}, vol. 62 (2000) 155–161.
\bibitem{Hanski94} I. Hanski, M. Kuussaari and M. Nieminen, Metapopulation structure and migration in the butterfly Melitaea cinixa, \emph{Ecology}, vol. 75 (1994) 747-762.
\bibitem{Turing52} A. Turing, The chemical basis of morphogenesis, \emph{Philosophical Transactions of the Royal Society B}, vol. 237 (1952) 37-72.
\bibitem{Randrup10} Y.A. Randrup and L. Stone, Connectivity, cycles, and persistence thresholds in metapopulation networks, \emph{PLoS Compuattional Biology}, vol. 6 (2010) e1000876.
\bibitem{Urban01} D. Urban and T. Keitt, Landscape connectivity: a graph-theoretic perspective, \emph{Ecology}, vol. 82 (2001) 1205-1218.
\bibitem{Pecora98} L. M. Pecora and T. L. Carroll, Master stability functions for synchronized coupled systems, \emph{Physical Review Letters}, vol. 80, (1998) 2109.
\bibitem{Tanner07} H.G. Tanner, A. Jadbabaie and G. J. Pappas, Flocking in fixed and switching networks, \emph{IEEE Transactions on Automatic Control}, vol. 52 (2007) 863-868.
\bibitem{Fax04} J.A. Fax and R. M. Murray, Information flow and cooperative control of vehicle formations, \emph{IEEE Transactions on Automatic Control}, vol. 49 (2004) 146-1476.
\bibitem{Merris94} R. Merris, Laplacian matrices of graphs: a survey, \emph{Linear Algebra and its Applications}, vol. 199 (1994) 143-176.
\bibitem{Mohar91} B. Mohar, The Laplacian spectrum of graphs, in graph theory, combinatorics, and applications,\emph{ Y. Alavi, G. Chartrand, O.R. Oellermann, A.J. Schwenk (Eds.), Proceedings of the Sixth Quadrennial International Conference on the Theory and Applications of Graphs, Western Michigan University, Kalamazoo, 1988, Wiley, New York (1991)}, 871-898.
\bibitem{Kernighan70} B. W. Kernighan and S. Lin, An efficient heuristic procedure for partitiong graphs, \emph{The Bell System Technical Journal}, (1970) 291-307.
\bibitem{Garey79} M.R. Garey and D.S. Johnson, Computers and Intractability: a guide to the theory of NP-Completeness, \emph{W. H. Freeman and Company, New York}, 1979.
\bibitem{Levins71} R. Levins and D. Culver, Regional coexistence of species and competition between rare species, \emph{Proceedings of the National Academy of Sciences}, vol. 68 (1971) 1246–1248.
\bibitem{Nakao10} H. Nakao and A.S. Mikhailov, Turing patterns in network-organized activator-inhibitor systems, \emph{Nature Physics}, vol. 6 (2010) 544-550.
\bibitem{Li10} M.Y. Li and Z. Shuai, Global stability problem for coupled systems of differential equations on networks, \emph{Journal of Differential Equations}, vol. 248 (2010) 1-20.
\bibitem{Ide16} Y. Ide, H. Izuhara and T. Machida, Turing instability in reaction-diffusion models on complex networks, \emph{Physica A}, vol. 457 (2016) 331–347.
\bibitem{Manyombe17} M.L.M. Manyombe, B. Tsanou J. Mbang and S. Bowong, A metapopulation model for the population dynamics of anopheles mosquito, \emph{Applied Mathematics and Computation}, vol. 307 (2017) 71-91.
\bibitem{Caswell78} H. Caswell, Predator mediated co-existence: a non-equilibrium model, \emph{The American Naturalist}, vol. 112 (1978) 127–154.
\bibitem{McCauley93} E. McCauley, W.G. Wilson and A.M. de Roos, Dynamics of age and spatially-structured predator–prey interactions: individual-based models and population level formulations, \emph{The American Naturalist}, vol. 142 (1993) 412–442.
\bibitem{Crowley81} P.H. Crowley, Dispersal and the stability of predator–prey interactions, \emph{The American Naturalist}, vol. 118 (1981) 673–701.
\bibitem{Kot01} M. Kot, Elements of Mathematical Ecology, \emph{Cambridge University Press}, 2001.

\bibitem{Anderson85} W. N. Anderson and T. D. Morley, Eigenvalues of the Laplacian of a graph, \emph{Linear and Multilinear Algebra}, vol. 18 (1985) 141-145.
\bibitem{Hahn97} G. Hahn and G. Sabidussi (eds.), Graph Symmetry: Algebraic Methods and Applications, \emph{Kluwer, Dordrecht}, 1997.

\bibitem{Ding07} J. Ding, Eigenvalues of rank-one updated matrices with some applications, \emph{Applied Mathematics Letters}, vol. 20 (2007) 1223-1226.
\bibitem{Jding07} J. Ding and G. Yao, The eigenvalue problem of a specially updated matrix, \emph{Applied Mathematics and Computation}, vol. 185 (2007) 415-420.
\bibitem{Merris98} R. Merris, Laplacian graph eigenvectors, \emph{Linear Algebra and its Applications}, vol. 278 (1998) 221-236.
\bibitem{Deo74} N. Deo, Graph Theory with Applications to Engineering and Computer Science, \emph{Prentice-Hall, Inc.}, 1974.
\bibitem{Schulz13} C. Schulz, High Quality Graph Partitioning, \emph{Dissertation}, 2013.
\bibitem{Bhatia97} R. Bahtia, Matrix Analysis, \emph{Springer-Verlag New York}, 1997.
\end{thebibliography}
\end{document}